\documentclass[leqno,12pt]{amsart} 
\usepackage[top=30truemm,bottom=30truemm,left=25truemm,right=25truemm]{geometry}
\usepackage{amssymb}
\usepackage{amsmath}
\usepackage{amsthm}
\usepackage{amscd}
\usepackage{mathrsfs}
\usepackage{graphicx}
\usepackage[dvips]{color}
\usepackage[all]{xy}

\usepackage{url}
\usepackage{comment}
\theoremstyle{plain} 
\newtheorem{theorem}{\indent\bf Theorem}[section]

\newtheorem{proposition}[theorem]{\indent\bf Proposition}

\theoremstyle{definition} 

\begin{document}

\title[Sharper $L^2$ extensions]{On sharper estimates of Ohsawa-Takegoshi $L^2$-extension theorem} 

\author[G. Hosono]{Genki Hosono} 

\subjclass[2010]{ 
	32A10, 32A07.
}
%
\keywords{ 
	Ohsawa-Takegoshi $L^2$ extension theorem, optimal estimate, complex Monge-Amp\`{e}re equation
}
\address{
Graduate School of Mathematical Sciences, The University of Tokyo \endgraf
3-8-1 Komaba, Meguro-ku, Tokyo, 153-8914 \endgraf
Japan
}
\email{genkih@ms.u-tokyo.ac.jp}
\maketitle
\begin{abstract}
We present an $L^2$-extension theorem with an estimate depending on the weight functions for domains in $\mathbb{C}$. When the Hartogs domain defined by the weight function is strictly pseudoconvex, this estimate is strictly sharper than known optimal estimates. When the weight function is radial, we prove that our estimate provides the $L^2$-minimum extension.
\end{abstract}

\section{Introduction}
The {\it Ohsawa-Takegoshi $L^2$-extension theorem} \cite{OT} states that, for a bounded pseudoconvex domain $\Omega$, a submanifold $V$ of $\Omega$ with certain conditions and a plurisubharmonic function $\varphi$ on $\Omega$, we can extend holomorphic functions on $V$ to ones on $\Omega$ with a priori $L^2$-estimates independent of $\varphi$.
This theorem and its generalizations are widely used in the studies of several complex variables and complex geometry.

Recently, the {\it optimal estimate} for the $L^2$-extension theorem was proved in \cite{Blo} and \cite{GZ}. Using this estimate, many problems including Suita conjecture were solved.
After that, a new proof of the optimal estimate was given in \cite{BL}.
Here we state the optimal result in the setting of \cite{BL}.

\begin{theorem}[Optimal estimate, {\cite{Blo}, \cite{GZ}, \cite{BL}}]\label{thm:optimal}
Let $\Omega \subset \mathbb{C}^n$ be a bounded pseudoconvex domain and $\varphi \in PSH(\Omega)$ be a plurisubharmonic function.
Let $V$ be a closed submanifold of $\Omega$ with codimension $k$.
Let $G$ be a negative plurisubharmonic function in $\Omega$ such that
$$G(z) \leq \log d^2_V(z) + A(z) \text{ on }\Omega,\text{ and} $$
$$G(z) \geq \log d^2_V(z) - B(z) \text{ near }V,$$
where $A(z)$ and $B(z)$ are continuous functions.
Then, for every $f\in\mathcal{O}(V)$ with $\int_V|f|^2 e^{-\varphi+kB} < +\infty $, there exists $F \in \mathcal{O}(\Omega) $ such that $F|_V = f$ and
$$\int_\Omega |F|^2 e^{-\varphi} \leq \sigma_k \int_V |f|^2 e^{-\varphi+kB}.$$
Here, $\sigma_k$ denotes the volume of the unit ball in $\mathbb{C}^k$.
\end{theorem}

In this paper, we call the function $G$ satisfying the conditions above a {\it Green-type function on $\Omega$ with poles along $V$}.

Note that this estimate does not depend on the weight functions $\varphi$. Thus we may sharpen the estimate if we allow the constant to depend on the weights. 
In this paper, using Theorem \ref{thm:optimal}, we prove a sharper estimate depending on weights $\varphi$ for domains in $\mathbb{C}$.

Let $\Omega$ be a bounded domain in $\mathbb{C}$ with $0 \in \Omega$.
For a subharmonic function $\varphi$ on $\Omega$, we will define a domain $\widetilde{\Omega}$ in $\mathbb{C}^2$ as
$$\widetilde{\Omega} := \{(z,w) \in \mathbb{C}^2: z \in \Omega, |w|^2 < e^{-\varphi(z)} \}. $$
Note that $\widetilde{\Omega}$ is pseudoconvex since $\varphi$ is subharmonic.
We will assume that there exists a Green-type function on $\widetilde{\Omega}$ with poles along $\{z=0\}$, i.e.\ there exists a negative plurisubharmonic function $\widetilde{G} = G_{\{z=0\}, \widetilde{\Omega}}$  with $\log |z|^2 + \widetilde{A}(z,w) \geq \widetilde{G} \geq \log |z|^2 - \widetilde{B}(z,w)$ for some continuous functions $\widetilde{A}$ and $\widetilde{B}$ on $\widetilde{\Omega}$.
Then our main theorem is as follows:
\begin{theorem}\label{thm:main_s}
Let $\Omega$ be a bounded domain in $\mathbb{C}$ and  $\varphi$ be a subharmonic function on $\Omega$ with $\varphi(0) = 0$.
Let $\widetilde{\Omega}$ and $\widetilde{B}$ as above.
Then, 
\begin{itemize}
	\item[(1)] there exists a holomorphic function $f \in \mathcal{O}(\Omega)$ such that $f(0) = 1$ and $$\int_\Omega|f(z)|^2 e^{-\varphi(z)} d\lambda(z) \leq \int_{|w|<1} e^{\widetilde{B}(0,w)} d\lambda(w).$$
	\item[(2)] Moreover, if $\widetilde{\Omega}$ is strictly pseudoconvex and $\Omega$ admits the Green function $G_{\Omega,0}$ with a pole at $0$, then (1) can be strictly sharper than the estimate in Theorem \ref{thm:optimal}. Precisely, if $\widetilde{\Omega}$ is strictly pseudoconvex, there exist functions $\widetilde{G}$, $\widetilde{A}$ and $\widetilde{B}$ satisfying the conditions above and
$$\int_{|w|<1} e^{\widetilde{B}(0,w)} d\lambda(w) < \pi e^{B(0)},$$
where $B(z) = G_{\Omega, 0} - \log |z|^2$ is a difference between the Green function on $\Omega$ and $\log |z|^2$.
\end{itemize}
\end{theorem}

When we take $\widetilde{G}$ as $\widetilde{G}(z,w) = G_{\Omega,0}(z)$, we have that
$$\int_{|w|<1} e^{\widetilde{B}(0,w)} d\lambda(w) = \pi e^{B(0)}.$$
In this case, (1) is the same to Theorem \ref{thm:optimal} in one-dimensional cases.

The proof of Theorem \ref{thm:main_s} (1) is a simple application of Theorem \ref{thm:optimal}. To prove (2), we will use the existence of solutions of complex Monge-Amp\`{e}re equations.


In radial cases, i.e.\ when $\Omega$ is the unit disc $\Delta \subset \mathbb{C}$ and the weight function $\varphi(z)$ depends only on $|z|$, we can prove that the Theorem \ref{thm:main_s} gives the value of the $L^2$-norm of the $L^2$-minimum extension.
These cases are treated in \S 3.
We do not know if Theorem \ref{thm:main_s} gives the minimum extension in general, even in the case where $\widetilde{\Omega}$ is strictly pseudoconvex.

In Theorem \ref{thm:main_s}, we use a Hartogs domain $\widetilde{\Omega}$ over $\Omega$. This kind of technique is also used in \cite{BL} and \cite{MV} to relax the condition $G<0$ on Green-type functions.

The organization of the paper is as follows. In \S \ref{sect: main}, we prove Theorem \ref{thm:main_s}.
In \S \ref{sect:radial}, we treat the case with radial weight functions.

\section{Proof of the Main Theorem}\label{sect: main}

\begin{proof}[Proof of Theorem \ref{thm:main_s} (1).]
We apply Theorem \ref{thm:optimal} to the submanifold $\{z=0\} \subset \widetilde{\Omega}$ using the Green-type function $\widetilde{G}$.
Here we use the trivial metric $e^{-\widetilde{\varphi}} \equiv 1$ on $\widetilde{\Omega}$.
Then we obtain an extension $F$ on $\widetilde{\Omega}$ of the constant function $1$ on $\{z=0\}$ such that
$$\int_{\widetilde{\Omega}} |F|^2 \leq \pi \int_{|w|<1} e^{B(0,w)}. $$
Consider the following function:
$$f(z) := F(z,0).$$
Then we have $f(0)= 1$ and, by the mean-value inequality, it follows that $$\pi\int_\Omega |f(z)|^2 e^{-\varphi(z)} \leq \int_{\widetilde{\Omega}} |F|^2.$$
Therefore $f$ satisfies the desired conditions.
\end{proof}

To prove (2), we use the theory of complex Monge-Amp\`{e}re equations to produce Green-type functions $\widetilde{G}$. The following existence and uniqueness result for the Dirichlet problem of the complex Monge-Amp\`{e}re equation was obtained by Bedford-Taylor:

\begin{theorem}[{\cite[Theorem D]{BT}}]\label{thm:solutions of complex MA equations}
	Let $\Omega \subset \mathbb{C}^n$ be a bounded strictly pseudoconvex domain, $\phi \in C(\partial \Omega)$, and $f \in C(\overline{\Omega})$, $f \geq 0$. Then there exists a unique function $u \in C(\overline{\Omega}) \cap PSH(\Omega)$ such that
	\begin{align*}
	(dd^c u)^n &= f  \text{ on }\Omega, \text{ and}\\
	u &=\phi \text{ on }\partial \Omega.
	\end{align*}
\end{theorem}

We also use the minimum principle:

\begin{theorem}[{\cite[Theorem A]{BT}}]\label{thm:minimum principle}
	Let $\Omega \subset \mathbb{C}^n$ be a bounded open set and $u,v \in C(\overline{\Omega}) \cap PSH(\Omega)$.
	When $(dd^c u)^n \leq (dd^c v)^n$ on $\Omega$, then it holds that
	$$\min_{\overline{\Omega}} (u-v) = \min_{\partial \Omega} (u-v). $$
	In other words, if $(dd^c u)^n \leq (dd^c v)^n$ on $\Omega$ and $u \geq v$ on $\partial \Omega$, then $u \geq v$ also on $\Omega$. 
\end{theorem}

\begin{proof}[Proof of Theorem \ref{thm:main_s} (2).]
Let us assume that $\widetilde{\Omega}$ is strictly pseudoconvex. 
By Theorem \ref{thm:solutions of complex MA equations}, there exists a plurisubharmonic function $\widetilde{u}$ on $\widetilde{\Omega}$ which is continuous on a neighborhood of $\overline{\widetilde{\Omega}}$ such that
\begin{align*}
(dd^c \widetilde{u})^2 &= 0 \hspace{4cm}\text{ on }\Omega, \text{ and}\\
\widetilde{u} &= -\max(\log|z|^2, C) \hspace{1cm}\text{ on }\partial\Omega,
\end{align*}
where $C$ is a sufficiently negative constant, which we will choose later in this proof.
We let $\widetilde{G} = \widetilde{G}_C := \log|z|^2 + \widetilde{u}$.
Then, on $\partial \Omega$, we have
$$\widetilde{G} = \log|z|^2 - \max(\log |z|^2, C) \leq \log |z|^2 - \log |z|^2 = 0. $$
Therefore, by the maximum principle, we have that $\widetilde{G} \leq 0$ on $\widetilde{\Omega}$.
By continuity of $\widetilde{u}$, we can take $\widetilde{A} = \widetilde{u}$ and $\widetilde{B} = -\widetilde{u}$.
We have to check
$$\int_{|w|<1} e^{-\widetilde{u}(0,w)} d\lambda(w) < \pi e^{B(0)}.$$
Recall that we write the Green function on $\Omega$ with a pole at 0 as
$$G_{\Omega, 0} = \log |z|^2 - B(z).$$
Note that $B$ is a harmonic function. We can take a sufficiently negative $C$ such that
$$\max(\log |z|^2, C) - B(z) < 0  \text{ on }\Omega.$$
Then, the function $\widetilde{u}'(z,w) := -B(z)$ as a function on $\widetilde{\Omega}$ satisfies the following conditions:
\begin{align*}
(dd^c \widetilde{u}')^2 = 0 \text{ on } \widetilde{\Omega}, \text{ and}\\
\widetilde{u}' = -B(z) \text{ on }\partial \widetilde{\Omega}.
\end{align*}
Since $-\max(\log |z|^2, C) \geq -B(z)$ on $\Omega$ (and thus on $\partial \widetilde{\Omega}$), Theorem \ref{thm:minimum principle} yields that $\widetilde{u} \geq \widetilde{u}'$ on $\widetilde{\Omega}$. In particular, it holds that $\widetilde{u}(0, w) \geq \widetilde{u}'(0,w) = -B(0)$ for $|w|<1$.
The boundary value of $\widetilde{u}$ on $\{z=0\} \cap \partial \widetilde{\Omega}$ is $-\max(\log 0 , C) = -C>B(0)$ and $\widetilde{u}$ is continuous on $\widetilde{\Omega}$ up to boundary, thus we have that the set $\{w: |w|<1, \widetilde{u}(0,w) > -B(0) \}$ has positive measure. Thus we have a strict inequality
$$ \int_{|w|<1} e^{-\widetilde{u}(0,w)} d\lambda(w) < \int_{|w|<1} e^{B(0)} d\lambda(w) = \pi e^{B(0)}.$$
\end{proof}

\section{Radial case}\label{sect:radial}
In radial cases, we can write $\widetilde{G}$ explicitly and prove that the estimate can be the best possible.
Assume that $\Omega = \Delta$ is the unit disc in $\mathbb{C}$ and $\varphi$ is radial, i.e.\ $\varphi(z)$ depends only on $|z|$.
Then we can write as $\varphi(z) = u(\log |z|^2)$ for a convex increasing function $u$ on $\mathbb{R}_{<0}$. First we assume that $u$ is strictly increasing and $\lim_{t \to -0} u(t) = +\infty$.
Define
$$\psi (w) := -u^{-1}(-\log|w|^2).$$
Then we can take $\widetilde{G}$ as follows:
$$\widetilde{G} = \log |z|^2 + \psi(w).$$
We prove the following
\begin{proposition}\label{prop:const}
It holds that
$$\int_\Delta e^{-\varphi(z)} d\lambda(z) = \int_{|w|<1} e^{-\psi(0,w)}. $$
\end{proposition}

\begin{proof}
We prove
$$ 2\pi \int_{r=0}^1 e^{-u(\log r^2)} rdr = 2\pi \int_{r=0}^1 e^{u^{-1}(-\log r^2)} rdr.$$
Using the substitution $2\log r = t$, this is equivalent to
$$\int_{t=-\infty}^0 e^{-u(t)} e^t dt = \int_{-\infty}^0 e^{u^{-1}(-t)}e^t dt.$$
In the right-hand side, letting $u^{-1}(-t) = s$, we have that
$$(RHS) = \int_0^{-\infty} e^{s}e^{-u(s)} (-u'(s))ds = \int_{-\infty}^0 e^{s-u(s)} u'(s) ds. $$
Then we can calculate the difference as follows: \begin{align*}
&\int_{t=-\infty}^0 e^{-u(t)} e^t dt -  \int_{s=-\infty}^0 e^{s-u(s)} u'(s) ds\\
&= \int_{t=-\infty}^0 (1-u'(t))e^{t-u(t)}dt \\
&= \int_{q=-\infty}^{-\infty} e^q dq = 0.
\end{align*}
In the last line, we use the substitution $t-u(t) = q$ and this integral equals to 0 because both endpoints of the interval of integration are $-\infty$.
\end{proof}
By Proposition \ref{prop:const} and Theorem \ref{thm:main_s}, when $u$ is strictly increasing and $u(-0) = +\infty$, there exists a holomorphic function $f$ on $\Delta$ such that $f(0) = 1$ and 
$$\int_\Delta |f|^2 e^{-\varphi}\leq \int_\Delta e^{-\varphi}.$$
For a general radial subharmonic function $\varphi$, apply this estimate to $\varphi - \epsilon \log(1-|z|^2)$ and take a limit $\epsilon \downarrow 0$.

When the weight is radial, the constant function 1 has the smallest $L^2$-norm with weight $\varphi$ among extensions of the function 1 on the subvariety $\{0\}$ (by the mean-value inequality). By Proposition \ref{prop:const}, we can show that the estimate in Theorem \ref{thm:main_s} provides the $L^2$-minimum extension in radial cases.

\vskip3mm
{\bf Acknowledgment. }
The author would like to thank Prof.\ Shigeharu Takayama, Prof.\ Takeo Ohsawa, Prof.\ Sachiko Hamano and Dr.\ Takayuki Koike for valuable comments and fruitful discussions.
Most of this paper was written at the time of the author's visit at the University of Stavanger. He is grateful to Prof.\ Alexander Rashkovskii for enormous supports.
The author is supported by Program for Leading Graduate Schools, MEXT, Japan.
He is also supported by the Grant-in-Aid for Scientific Research (KAKENHI No.15J08115).

\bibliographystyle{plain}

\end{document}